\documentclass[reqno,12pt,letterpaper]{amsart}
\usepackage{amsmath,amssymb,amsthm,graphicx,mathrsfs,url}
\usepackage[applemac]{inputenc}
\usepackage[usenames,dvipsnames]{color}
\usepackage[colorlinks=true,linkcolor=Red,citecolor=Green]{hyperref}
\usepackage{enumitem}
\usepackage[russian, francais, english]{babel}
\usepackage[dvipsnames]{xcolor}
\usepackage{csquotes}
\usepackage[all]{xy}

\def\?[#1]{\textbf{[#1]}\marginpar{\Large{\textbf{??}}}}

\newtheorem{theo}{Theorem}
\newtheorem{prop}{Proposition}[section]

\newtheorem{lemm}[prop]{Lemma}
\theoremstyle{definition}
\newtheorem{corr}[prop]{Corollary}

\newtheorem{rem}[prop]{Remark}

\numberwithin{equation}{section}

\newcommand{\R}{\mathbb{R}}
\newcommand{\N}{\mathbb{N}}
\renewcommand{\C}{\mathbb{C}}

\newcommand{\Z}{\mathbb{Z}}

\newcommand{\vol}{\mathrm{vol}}

\newcommand{\dd}{\mathrm{d}}
\newcommand{\Lie}{\mathcal{L}}

\newcommand{\ind}{\mathrm{ind}}

\let\Re=\Real

\DeclareMathOperator{\supp}{supp}

\DeclareMathOperator{\WF}{WF}

\DeclareMathOperator{\id}{Id}

\title{Poincar\'e series for surfaces with boundary}
\author[Y.~Chaubet]{Yann Chaubet}
\address{Institut de Math\'ematiques d'Orsay, Facult\'e des sciences d'Orsay, Universit\'e Paris-Saclay, 91405 Orsay, France}
\email{yann.chaubet@universite-paris-saclay.fr}

\begin{document}

\maketitle

\begin{abstract}
We provide a meromorphic continuation for Poincar\'e series counting orthogeodesics of a negatively curved surface with totally geodesic boundary, as well as for Poincar\'e series counting geodesic arcs linking two points. For the latter series, we show that the value at zero coincides with the inverse of the Euler characteristic of the surface.
\end{abstract}

\section{Introduction}
Let $(\Sigma, g)$ be a connected oriented negatively curved surface with totally geodesic boundary $\partial \Sigma$. We denote by $\mathcal{G}^\perp$ the set of orthogeodesics of $\Sigma,$ that is, the set of geodesics $\gamma : [0, \ell] \to \Sigma$ (parametrized by arc length) such that $\gamma(0), \gamma(\ell) \in \partial \Sigma$, $\gamma'(0) \perp T_{\gamma(0)} \partial \Sigma$ and $\gamma'(\ell) \perp T_{\gamma(\ell)}\partial \Sigma$. For large $\Re(s)$ the Poincar\'e series
\begin{equation}\label{eq:defeta}
\eta(s) = \sum_{\gamma \in \mathcal{G}^\perp} {e}^{-s\ell(\gamma)},
\end{equation}
where $\ell(\gamma)$ denotes the length of the geodesic arc $\gamma$, {is well defined (see \S\ref{subsec:pairing}). In this paper we will prove the following} 
\begin{theo}\label{thm:main}
The Poincar\'e series $s \mapsto \eta(s)$ extends meromorphically to the whole complex plane and vanishes at $s=0$.
\end{theo}
If $x \neq y \in \Sigma$, we may also consider the Poincar\'e series associated to the geodesic arcs joining $x$ to $y$. Namely, we set for $\Re(s)$ large enough
$$
\eta_{x,y}(s) = \sum_{\gamma : x \rightsquigarrow y} {e}^{-s\ell(\gamma)},
$$
where the sum runs over all geodesic arcs $\gamma : [0, \ell(\gamma)] \to \Sigma$ (parametrized by arc length) such that $\gamma(0) = x$ and $\gamma(\ell) = y$. {Then we have the following result.}

\begin{theo}\label{thm:points}
The Poincar\'e series $s \mapsto \eta_{x,y}(s)$ extends meromorphically to the whole complex plane and 
$$
\eta_{x,y}(0) = \frac{1}{\chi(\Sigma)},
$$
where $\chi(\Sigma)$ is the Euler characteristic of $\Sigma$.
\end{theo}

{To the best of our knowledge, Theorem \ref{thm:main} is the first result on a series involving the orthospectrum (that is, the set of lengths of orthogeodesics) of a surface with totally geodesic boundary which has \textit{variable} negative curvature.} For hyperbolic surfaces (i.e. surfaces with \textit{constant} curvature $-1$) with totally geodesic boundary, the orthospectrum has been studied by many authors, among others Basmajian \cite{basmajian1993orthogonal}, Bridgeman \cite{bridgeman2011orthospectra}, Calegari \cite{calegari2010chimneys} (see also Bridgeman--Kahn \cite{bridgeman2010hyperbolic}). 
In particular they show that if $(\Sigma, g)$ is a compact hyperbolic surface with totally geodesic boundary, one has
$$
\ell(\partial \Sigma) = \sum_{\gamma \in \mathcal{G}^\perp} 2 \log \coth(\ell(\gamma) / 2), \quad \mathrm{vol}(\Sigma) = \frac{2}{\pi} \sum_{\gamma \in \mathcal{G}^\perp} \mathcal{R}\left(\mathrm{sech}^2(\ell(\gamma) / 2)\right),
$$
where $\ell(\partial \Sigma)$ is the length of the boundary of $\Sigma$, $\vol(\Sigma)$ is the area of $\Sigma$ and $\mathcal{R}$ is the Rogers dilogarithm function. We refer to \cite{bridgeman2016identities} for a detailed exposition of those results.

{In order to study the Poincar\'e series $\eta(s)$ and $\eta_{x,y}(s)$, we will adopt the elegant approach of Nguyen Viet Dang and Gabriel Rivi\`ere \cite{dang2020poincar}, which consists in interpreting both series as distributional pairings involving the resolvent of the geodesic flow. On a \textit{closed} surface with negative curvature, Dang and Rivi\`ere proved that Poincar\'e series associated to orthogeodesic arcs joining any two homologically trivial closed geodesics, as well as Poincar\'e series associated to geodesic arcs linking two points, admit a meromorphic extension to the whole complex plane; moreover they computed their values at zero --- for the series associated to geodesic arcs linking two points, they found (as here) that this value coincides with the inverse of the Euler characteristic of the surface. The main novelty of this article is that we work in the open case, which leads us to use the theory of Pollicott--Ruelle resonances for open systems developed by Dyatlov and Guillarmou \cite{dyatlov2016pollicott}, as well as a result of Hadfield \cite{hadfield2018zeta} about the topology of resonant states for surfaces with boundary. We refer to \cite{dang2020poincar} for precise references about Poincar\'e series counting geodesic arcs.
}

\subsection*{Acknowledgements} I thank Nguyen Viet Dang, Colin Guillarmou and Gabriel Rivi\`ere for fruitful discussions. I also thank the anonymous referees for interesting suggestions that helped improve the quality of this manuscript. This project has received funding from the European Research Council (ERC) under the European Unions Horizon 2020 research and innovation programme (grant agreement No. 725967).

\section{Geometrical and dynamical preliminaries}

We introduce in this section the main tools that will help us to understand $\eta$ and $\eta_{x,y}$.

\subsection{Extension to a surface with strictly convex boundary}

We extend $(\Sigma,g)$ into a slightly larger negatively curved surface with boundary $(\Sigma',g')$. We take $\delta > 0$ small and we set
$$
\Sigma_\delta = \{x \in \Sigma' {~:~}\mathrm{dist}_{g'}(x, \Sigma) < \delta\}.
$$
Then since $\partial \Sigma$ is totally geodesic and $(\Sigma', g')$ is negatively curved, it follows that $\Sigma_\delta$ has strictly convex boundary, in the sense that the second fundamental form of $\partial \Sigma_\delta$ with respect to the outward normal vector field is positive ({see for example \cite[Lemma 2.4]{chaubet2021closed}}). We denote by
$$
M_\delta = S\Sigma_\delta = \{(x, v) \in T\Sigma_\delta {~:~}\|v\|_g = 1\}
$$
the unit tangent bundle of the surface $\Sigma_\delta$, and {by} $\pi : M_\delta \to \Sigma_\delta$ the natural projection.

\subsection{Structural forms}\label{subsec:structuralforms}
We recall here some classical facts from \cite[\S7.2]{singer1976lecture} about geometry of surfaces. We have the Liouville one-form $\alpha$ on $M_\delta$ defined by
$$
\langle \alpha(z), w \rangle = \langle \dd_{(x,v)} \pi(w), v \rangle, \quad {z = } (x,v) \in M_\delta, \quad w \in T_{(x,v)}M_\delta. 
$$
Then $\alpha$ is a contact form (that is, $\alpha \wedge \dd \alpha$ is a volume form on $M_\delta$) and it turns out that the geodesic vector field $X$ is the Reeb vector field associated to $\alpha$, that is, it satisfies
$$
\iota_X \alpha = 1, \quad \iota_X \dd \alpha = 0,
$$ 
where $\iota$ denote the interior product. We  set $\beta = R_{\pi/2}^*\alpha$ where for $\theta \in \R$, we denote by $R_{\theta} : M_\delta \to M_\delta$ the rotation of angle $\theta$ in the fibers {(which is defined thanks to the orientation of $\Sigma$)}. Then the volume form $\vol_g$ of $\Sigma_\delta$ satisfies \cite[p. 166]{singer1976lecture}
\begin{equation}\label{eq:volume}
\pi^*\vol_g = \alpha \wedge \beta.
\end{equation}
We denote by $\psi$ the connection one-form (see \cite[Theorem p.169]{singer1976lecture}), that is, the unique one-form on $M_\delta$ satisfying
\begin{equation}\label{eq:structural}
\iota_V \psi = 1, \quad \dd \alpha = \psi \wedge \beta, \quad \dd \beta = \alpha \wedge \psi, \quad \dd \psi = -(\kappa \circ \pi) \alpha \wedge \beta,
\end{equation}
where $V$ is the vector field generating $(R_\theta)_{\theta \in \R}$ and $\kappa$ is the Gauss curvature of $\Sigma$. Then $(\alpha, \beta, \psi)$ is a global frame of $T^*M_\delta$. We denote {by} $H$ the vector field on $M_\delta$ such that $(X, H, V)$ is the dual frame of $(\alpha, \beta, \psi)$. We then have the following commutation relations \cite[p. 170]{singer1976lecture}
\begin{equation}\label{eq:commutation}
[V,X] = H, \quad [V,H] = -X, \quad [X, H] = (\kappa \circ \pi) V.
\end{equation}
The orientation of $M_\delta$ will be chosen so that $(X, H, V)$ is positively oriented. {On $\partial M$, we have a precise description $(X, H, V)$, as follows (see \cite[Lemma 2.2]{chaubet2021closed}).
\begin{lemm}\label{lem:coord}
Let $\gamma_\star$ be a connected component of $\partial \Sigma$ (which is the image of a closed geodesic) and denote by $\ell_\star > 0$ its length. Then there is a tubular neighborhood $U$ of $\pi^{-1}(\gamma_\star)$ and coordinates $(\rho, \tau, \theta)$ on $U$ with
$$
U \simeq (-\delta, \delta)_\rho \times (\R/\ell_\star \Z)_\tau \times (\R/2 \pi\Z)_\theta,
$$
and such that
$$
|\rho(z)| = \mathrm{dist}(z, \gamma_\star), \quad S_z \Sigma = \{(\rho(z), \tau(z), \theta)~:~\theta \in \R/2\pi\Z\}, \quad z \in U.
$$
Moreover in these coordinates we have, on $\{\rho = 0\} \simeq \gamma_\star$,
$$
X(z) = \cos(\theta) \partial_\tau + \sin(\theta) \partial_\rho, \quad H = -\sin(\theta)\partial_\tau + \cos(\theta) \partial_\rho, \quad V = \partial_\theta.
$$
\end{lemm}
}

\subsection{Extension of the geodesic vector field}
{We embed $M_\delta$ into a compact manifold without boundary $N$ (for example by taking the doubling manifold); then by \cite[{Lemma 2.1}]{dyatlov2016pollicott}, we may extend the geodesic vector field $X$ to a vector field on $N$} so that $M_\delta$ is convex {with respect to $X$} in the sense that for any $T \geqslant 0$, 
\begin{equation}\label{eq:convex}
x, \varphi_T(x) \in M_\delta \quad  \implies \quad \forall t \in [0,T], \quad \varphi_t(x) \in M_\delta,
\end{equation}
{where $\varphi_t$ is the flow induced by $X$}.
Let $\rho_\delta \in C^\infty(M_\delta, [0,1])$ be a boundary defining function for $M_\delta$, that is, $\rho_\delta > 0$ on $M_\delta \setminus \partial M_\delta$, $\rho_\delta = 0$ on $\partial M_\delta$ and $\dd \rho_\delta \neq 0$ on $\partial M_\delta$ (for example we can take $\rho_\delta(x,v) = \mathrm{dist}(x, \partial\Sigma_\delta)$). Then the strict convexity of $\partial \Sigma_\delta$ implies that $\partial M_\delta$ is strictly convex in the sense that for every $x \in \partial M_\delta$ one has 
\begin{equation}\label{eq:strictconvex}
X \rho_\delta(x) = 0 \quad \implies \quad X^2\rho_\delta(x) < 0
\end{equation}
(see \cite[\S6.3]{dyatlov2016pollicott}).

\subsection{Hyperbolicity of the geodesic flow}
We define
$$
\Gamma_\pm = \{z \in M_\delta {~:~} \varphi_{\mp t}(z) \in M_\delta,~t \geqslant 0\}, \quad K = \Gamma_+\cap\Gamma_- \subset M.
$$
By \cite[\S3.9 and Theorem 3.2.17]{klingenberg2011riemannian} the geodesic flow $(\varphi_t)$ is hyperbolic on $K$, that is, for every $z \in K$ there is a decomposition
$$
T_z M_\delta = E_s(z) \oplus E_u(z) \oplus \R X(z)
$$
depending continuously on $z$, which is invariant by $\dd \varphi_t$ and such that, for some $C, \nu > 0$,
$$
\left\|\dd \varphi_t(z) w\right\| \leqslant C {e}^{-\nu t} \|w\| , \quad w \in E_s(z), \quad t \geqslant 0,
$$
and 
$$
\left\|\dd \varphi_{-t}(z) w\right\| \leqslant C {e}^{-\nu t} \|w\|, \quad w \in E_u(z), \quad t \geqslant 0.
$$
Moreover, by \cite[Lemma 2.10]{dyatlov2016pollicott}, there are two vector subbundles $E_{\pm} \subset T_{\Gamma_\pm} M_\delta$ {(here $T_{\Gamma_\pm} M_\delta = TM_\delta|_{\Gamma_\pm}$)} with the following properties:
\begin{enumerate}
\item $E_+|_K = E_u$ and $E_-|K = E_s$ and $E_\pm(z)$ depends continuously on $z \in \Gamma_\pm$;

\item \label{it:2} $\langle \alpha, E_\pm \rangle  = 0$;

\item\label{it:3} For some constants $C',\nu' > 0$ we have
$$
\|\dd \varphi_{\mp t} (z)w\| \leqslant C' {e}^{-\nu' t}\|w\|, \quad w \in E_\pm(z), \quad z \in {\Gamma_\pm}, \quad t \geqslant 0;
$$
\item If $z \in \Gamma_\pm$ and $w \in T_z{M_\delta}$ satisfy $w \notin {\R X(z) \oplus }E_\pm(z)$, then as $t \to \mp \infty$
$$
\left\|\dd \varphi_{t}(z)w \right\| \rightarrow \infty,\quad \frac{\dd \varphi_t(z)w}{\left\|\dd \varphi_t(z)w\right\|} \rightarrow E_{\mp}|_K.
$$
\end{enumerate}
Moreover, we have the following description of $E_\pm.$ 
{
\begin{lemm}\label{lem:rpm}
There are continuous functions $r_\pm : \Gamma_\pm \to \R$ such that $\pm r_\pm > 0$ on $\Gamma_\pm$ and 
\begin{equation}\label{eq:e_pm}
E_\pm(z) = \mathbb{R}\left(H(z) + r_\pm(z) V(z)\right), \quad z \in \Gamma_\pm.
\end{equation}
\end{lemm}
}
\begin{proof}
{As the contact form $\alpha$ is preserved by the flow $\varphi_t$, we get by property \eqref{it:2} above $E_\pm(z) \subset \ker \alpha(z) = \R(H(z) \oplus V(z))$. In a first step, we will assume that $E_\pm(z) \cap \R V(z) = \{0\}$ for every $z \in \Gamma_\pm$ (we shall prove it later). Since the bundles $E_\pm$ are continuous, we deduce that there are two continuous functions $r_\pm : \Gamma_\pm \to \R$ such that \eqref{eq:e_pm} holds. 
}

{
Next, we show that $\pm r_\pm > 0$. The fact that $\dd \varphi_{\mp t}(z) E_\pm(z) \subset E_\pm(z)$ for $t \geqslant 0$ implies that the map $t \mapsto r_\pm(\varphi_{\mp t}(z))$ is smooth on $\R_+$ for any $z \in \Gamma_\pm$ (since $\R(H \oplus V)$ is preserved by $\dd \varphi_t$). We may thus compute, on $\Gamma_\pm$,
\begin{align}\label{eq:commutation2}
[X, H + r_\pm V] &= [X, H] + (X r_\pm) V + r_\pm [X, V] = \left(\kappa \circ \pi + X r_\pm\right) V - r_\pm H,
\end{align}
where we used the commutation relations \eqref{eq:commutation}.
As $E_\pm$ is preserved by the flow, we must have $[X, H+r_\pm V] \in E_\pm$; thus combining \eqref{eq:commutation2} and \eqref{eq:e_pm} we obtain the following Riccati equation:
\begin{equation}\label{eq:riccati}
X r_\pm + r_\pm^2 + \kappa \circ \pi = 0 \quad \text{ on } \Gamma_\pm.
\end{equation}
We now prove that $r_+ > 0$. Let $z \in \Gamma_+$ and set $U(t) = (H + r_+ V)(\varphi_{-t}(z)) \in E_+(\varphi_{-t}(z))$ and $r_+(t) = r_+(\varphi_{-t}(z))$ for $t \geqslant 0$. By \eqref{eq:commutation2} and \eqref{eq:riccati} we have $[-X, U(t)] = r_+(t) U(t)$ and thus
\begin{equation}\label{eq:contraction}
\dd \varphi_{-t}(z) U(0) = \exp \left(- \int_{0}^t r_+(u) \dd u\right) U(t), \quad t \geqslant 0.
\end{equation}
On the other hand, equation \eqref{eq:riccati} implies that for any $t \geqslant 0$ we have the implication
$$
r_+(t) = 0 \quad \implies \quad r_+'(t) = -X r_+(\varphi_{-t}(z)) < 0
$$
since $\kappa < 0$ everywhere. Therefore, if $r_+(t) \leqslant 0$ for some $t$, then $r_+(u) \leqslant 0$ for all $u \geqslant t.$ This is not possible by \eqref{eq:contraction} since $\dd \varphi_{-t}(z)U(0) \to 0$ as $t \to +\infty$. We therefore proved that $r_+(t) > 0$ for all $t \geqslant 0$.  Thus $r_+ > 0$ on $\Gamma_+$ and similarly, one can show that $r_-(z) < 0$ for all $z \in \Gamma_-.$}

{
It remains to prove that $E_\pm(z) \cap \R V(z) = \{0\}$ for any $z \in \Gamma_\pm$. Let $z \in \Gamma_+$, and write $V(t) = V(\varphi_{-t}(z))$ and $H(t) = H(\varphi_{-t}(z))$ for $t \geqslant 0$. Then there are smooth functions 
$a,b : [0, \infty[ \to \R$ such that for any $t \geqslant 0$ one has
$$
\dd \varphi_{-t}(z)V(z) = a(t) H(t) + b(t) V(t)
$$
The commutation relations \eqref{eq:commutation} imply that
$$
a'(t) + b(t) = 0, \quad  \kappa(t) a(t) + b'(t) = 0, \quad t \geqslant 0,
$$
where $\kappa(t) = (\kappa \circ \pi \circ \varphi_{-t})(z)$. Thus $a''(t) + \kappa(t)a(t) = 0$; moreover we have $a(0) = 0$ and $a'(0) = -b(0) = -1$; from this it is easy to deduce that $a'(t) < 0$ for every $t \geqslant 0$. In particular there are $C_1, C_2 > 0$ such that $a(t) \leqslant -C_1$ for every $t \geqslant C_2$ and thus for some $C > 0$ we have $\|\dd \varphi_{-t}(z)V(z)\| \geqslant C$ for any $t \geqslant 0$. As a consequence, we obtain that $V(z) \notin E_+(z).$ Similarly, one can prove that $V(z) \notin E_-(z)$ for any $z \in \Gamma_-$. This concludes the proof of the lemma.
}
\end{proof}

{
\begin{rem}\label{rem:neq}
Looking carefully at the proof of Lemma \ref{lem:rpm}, we see that for any $z$ and $t$ such that $\varphi_t(z) \in M_\delta$ we have
\begin{equation}\label{eq:positive}
\pm \langle \varphi_{t}^*\beta(z), V(z)\rangle > 0 \quad \text{ and } \quad \pm \langle \varphi_t^*\psi(z), H(z) \rangle > 0
\end{equation}
whenever $\pm t > 0$. Indeed, the first part of \eqref{eq:positive} follows from the fact that, with the notations of the proof of Lemma \ref{lem:rpm}, one has $a(t) = \langle \varphi_{-t}^*\beta(z), V(z) \rangle < 0$ for $t > 0$ (since $a'(t) < 0$ and $a(0) = 0$), and reversing the time we get that $a(t) > 0$ whenever $t < 0$. The second part of \eqref{eq:positive} was not explicitly proven but the proof is very similar.
\end{rem}
}

\subsection{The resolvent}\label{subsec:resolvent}
For $\Re(s)$ large enough, consider the operator $R(s)$ defined on $\Omega^\bullet(N)$ by
\begin{equation}\label{eq:formularesolv}
R(s) = \int_{0}^{+\infty} {e}^{-ts} \varphi_{-t}^* \dd t.
\end{equation}
Here $\Omega^\bullet(N)$ denotes the space of smooth differential forms on $N$. Then it holds 
$$
(\Lie_X +s) R(s) = \id_{\Omega^\bullet(N)} = R(s) (\Lie_X + s).
$$
Let $\chi \in C^\infty_c(M_\delta \setminus \partial M_\delta)$ such that $\chi \equiv 1$ on {$M_{\delta/2}$}, and let 
$$Q(s) = \chi R(s) \chi.$$
Then it follows from \cite[Theorem 1]{dyatlov2016pollicott} that the family of operators $s \mapsto \chi R(s) \chi$ extends to a family of operators
$$
Q(s) : \Omega_c^\bullet(M_\delta^\circ) \to \mathcal{D}'^\bullet(M_\delta^\circ)
$$
meromorphic in $s \in \C$, which satisfies, for $w \in \Omega^\bullet_c(M^\circ_\delta)$ supported in $\{\chi = 1\}$,
\begin{equation}\label{eq:inverse}
(\Lie_X + s)Q(s)w = w \quad \text{ on } \quad \{\chi = 1\},
\end{equation}
for any $s \in \C$ which is not a pole of $s \mapsto Q(s)$. Here, $M_\delta^\circ$ denotes the interior of $M_\delta$ and if $U$ is a manifold, $\Omega^\bullet_c(U)$ denotes the space of compactly supported differential forms on $U$ while $\mathcal{D}'^\bullet(U)$ denote its dual space, that is, the space of currents.
{
In what follows, for any distribution $A \in \mathcal{D}'(T^*M_\delta \times T^*M_\delta)$, we will set
$$
\WF'(A) = \{(z, \xi, z', \xi') \in T^*(M_\delta \times M_\delta) {~:~}(z, \xi, z', -\xi') \in \WF(A)\},
$$
where $\WF$ is the H\"ormander wavefront set, see \cite[\S8]{hor1}.}  The microlocal structure of $Q(s)$ is given by ({see \cite[Lemma 4.5]{dyatlov2016pollicott}}, in what follows we identify $Q(s)$ and its Schwartz kernel)
\begin{equation}\label{eq:wf}
\WF'(Q(s)) \subset \Delta(T^*M_\delta) \cup \Upsilon_+ \cup (E_+^* \times E_-^*)
\end{equation}
where $\Delta(T^*M_\delta) = \{(\xi, \xi) {~:~}\xi \in T^*M_\delta \} \subset T^*(M_\delta \times M_\delta)$ and 
$$
\Upsilon_+ = \{(\Phi_t(z, \xi), (z,\xi)) {~:~}t \geqslant 0,~\langle \xi, X(z) \rangle = 0,~z \in M_\delta,~\varphi_t(z) \in M_\delta\}.
$$
Here $\Phi_t$ denotes the symplectic lift of $\varphi_t$ on $T^*M_\delta$, that is 
$$
\Phi_t(z, \xi) = (\varphi_t(z), (\dd_z\varphi_t)^{-\top} \xi), \quad (z, \xi) \in T^*M_\delta, \quad \varphi_t(z) \in M_\delta,
$$ 
and the subbundles $E_\pm^* \subset T^*_{\Gamma_\pm}M_\delta$ are defined by 
$E_\pm^*(\R X(z) \oplus E_\pm) = 0.$ 
In particular, we have
$$
\|\Phi_{\mp t}(z, \xi)\| \to + \infty, \quad (z, \xi) \in E_\pm^*, \quad t \to +\infty
$$
and 
$$
E_\pm^*(z) =\R\left( r_\pm(z) \beta(z) - \psi(z)\right), \quad z \in \Gamma_\pm.
$$

\section{Poincar\'e series}

In this section, we give a description of the Poincar\'e series $\eta(s)$ in terms of a pairing involving the operator $Q(s).$

\subsection{Counting measure}
{Let} $\Lambda, \bar \Lambda \subset M_\delta$ be the one-dimensional submanifolds of $M_\delta$ defined by
$$
\Lambda = \{(x,\nu(x)) {~:~}x \in \partial \Sigma\}, \quad \bar \Lambda = \{(x, -\nu(x)) {~:~}x \in \partial \Sigma\}.
$$
where $\nu : \partial \Sigma \to M$ is the outward normal pointing vector to $\partial \Sigma$. {Those manifolds are oriented according to the orientation of $\partial \Sigma$ {which is itself oriented by $\partial_\tau$ in the coordinates of Lemma \ref{lem:coord}}; note also that in the coordinates given by Lemma \ref{lem:coord} we have (here $\gamma_\star$ is a connected component of $\partial \Sigma$)
\begin{equation}\label{eq:lambdacoord}
\Lambda|_{\gamma_\star} = \{(0, \tau, \pi / 2)~:~ \tau \in \R/\ell_\star \Z\}, \quad \bar \Lambda|_{\gamma_\star} = \{(0, \tau, -\pi / 2)~:~ \tau \in \R/\ell_\star \Z\};
\end{equation}
in particular it holds
\begin{equation}\label{eq:tangent}
T_z \Lambda = \R H(z), \quad T_{z'} \bar \Lambda = \R H(z'), \quad (z, z') \in \Lambda \times \bar \Lambda.
\end{equation}
}
For $\tau \geqslant 0$ {and $z \in \Lambda$} such that $\varphi_{-\tau}(z) \in \bar \Lambda,$ we will set
$
\varepsilon(\tau,z) = 1
$
if 
\begin{equation}\label{eq:directsum}
T_z \Lambda \oplus \R X(z) \oplus \dd_{\varphi_{-{\tau}}(z)} \varphi_{{\tau}}\left(T_{\varphi_{-{\tau}}({z})}\bar \Lambda\right)
\end{equation}
has the same orientation {as} $TM_\delta$, and $\varepsilon(\tau, z) = -1$ otherwise {(note that the sum \eqref{eq:directsum} is always direct as the component of $\dd \varphi_\tau(z') H(z')$ on $V(\varphi_\tau(z'))$ is positive by Remark \ref{rem:neq}, since $\varphi_\tau(\bar \Lambda) \cap \Lambda \neq \emptyset$ implies $\tau > 0$).}
{
\begin{lemm}For any $\tau \geqslant 0$ and $z \in \Lambda$ such that $\varphi_{-\tau}(z) \in \bar \Lambda$ it holds
$$
\varepsilon(\tau, z) = 1.
$$
\end{lemm}
}

\begin{proof}
Indeed, $\bar \Lambda$ is oriented so that {$\det_{T_{z'} \bar \Lambda}H(z') > 0$} for $z' \in \bar \Lambda$; moreover the component of $\dd \varphi_{t}(z)H(z)$ on $V(\varphi_t(z))$ is positive (meaning that $\langle \psi(\varphi_t(z)), \dd_{z} \varphi_{t}H(z)\rangle > 0$) whenever $t > 0$ (see Remark \ref{rem:neq}). Thus, since $T_z\Lambda$ is oriented so that ${\det_{T_z\Lambda}H(z) < 0}$ we obtain  $\varepsilon(\tau, z)$ is equal to the sign of $\det_{T_{z}M_\delta}(-H, X, H + f(z, \tau)V)$ for some $f(z, \tau) > 0$, which is $1$ as $(X, H, V)$ is positively oriented.
\end{proof}

{
In what follows, if $P$ is an embedded, oriented, compact, $k$-dimensional submanifold of $N$, we will denote by $[P] \in \mathcal{D}'^{n-k}(N)$ the associated integration current, which is defined by
$$
\int_N [P]\wedge \omega = \int_P {\iota_P}^*\omega, \quad \omega \in \Omega^k(N),
$$
where $\iota_P : P \hookrightarrow N$ is the inclusion.
} We then have the following geometrical lemma, which is a {direct} adaptation of \cite[Lemma 4.11]{dang2020poincar} in our context.

\begin{lemm}\label{lem:distribution}
The expression
$$
\mu(t) = \sum_{\substack{\tau \geqslant 0 \\ \Lambda \cap \varphi_\tau(\bar \Lambda) \neq \emptyset}} \left(\sum_{z \in \Lambda \cap \varphi_\tau(\bar \Lambda)} \varepsilon(\tau, z)\right)
\delta(t-\tau),$$
makes sense and defines a distribution $\mu \in \mathcal{D}'(\R_{>0})$. Moreover, it coincides with 
$$
t \mapsto - \int_{N} [\Lambda] \wedge \left(\iota_X \varphi_{-t}^*[\bar \Lambda]\right).
$$

\end{lemm}
{\begin{rem}
\begin{enumerate}[label=(\roman*)]
\item Lemma \ref{lem:distribution} can be reformulated as follows. For any $\chi \in C^\infty_c(\R_+)$, the product 
$$A_\chi = [\Lambda]\wedge \int_{\R_+} \chi(t) \iota_X \varphi_{-t}^* [\bar \Lambda] \dd t$$
is well defined and $\langle 1, A_\chi \rangle = - \langle \mu, \chi\rangle$ (here the first pairing takes place on $N$ while the second one takes place on $\R_+$).
\item Lemma \ref{lem:distribution} is an elementary result coming from the theory of currents and is not specific to $(\Lambda, \bar \Lambda,\varphi_t).$ Indeed, this lemma will hold true for if we replace $\Lambda, \bar \Lambda$ and the flow $\varphi_t$ by arbitrary submanifolds $N_1, N_2$ and another flow $\psi_t$, whenever the sum \eqref{eq:directsum} is direct (replacing $(\Lambda, \bar \Lambda, \varphi_t)$ by $(N_1, N_2, \psi_t)$) and $\dim N_1 + \dim N_2 + 1 = \dim N$; we refer to \cite[Lemma 4.11]{dang2020poincar} for more details.
\end{enumerate}
\end{rem}}
\begin{proof}
We note that $\Lambda \cap \bar \Lambda = \emptyset$, and $X(z) \notin T_z\bar \Lambda$ for any $z \in \bar \Lambda$. Moreover, it holds $\dim(\Lambda) + \dim(\bar \Lambda) + 1 = \dim(N)$. Hence by \eqref{eq:directsum} we can apply \cite[Lemma 4.11]{dang2020poincar} to obtain the sought result (note however that here the vector field $X$ on $N$ may have singular points, but this is not a problem {since the proof of \cite[Lemma 4.11]{dang2020poincar} is local in nature and the} singular points are far away from $\Lambda$).
\end{proof}

\subsection{A pairing formula for the Poincar\'e series}\label{subsec:pairing}

Note that (\ref{eq:wf}) implies that $Q(s)\iota_X[\bar \Lambda]$ is well defined. {Indeed, according to \cite[Theorem 8.1.9]{hor1} one has}
$
\WF(\iota_X [\bar \Lambda]) \subset N^* \bar \Lambda
$
where
$$
N_z^*\bar \Lambda = \{\xi \in T^*_zM {~:~}\langle \xi, H(z) \rangle = 0\} = \R \alpha(z) \oplus \R \psi(z),
$$
where $\psi$ is the connection form. For $z \in \Lambda$ we have $T_z \Lambda = \R H(z)$ and thus
\begin{equation}\label{eq:n*lambda}
N^*_z \Lambda = \R\alpha(z) \oplus \R\psi(z),
\end{equation}
(of course the same formula holds if we replace $\Lambda$ by $\bar \Lambda$). We have $E_-^* \cap  N^*\bar \Lambda \subset \{0\} $, since for $z \in \Gamma_- \cap \bar \Lambda$ we have $E_-^*(z) = \R (r_-(z)\beta(z) - \psi(z))$, and $r_-(z) < 0$ by Lemma \ref{lem:rpm}.
By (\ref{eq:wf})
we can apply \cite[Theorem 8.2.13]{hor1} to see that $Q(s)\iota_X [\bar \Lambda]$ is well defined, and 
\begin{equation}\label{eq:wfqss}
\WF(Q(s) \iota_X[\bar \Lambda]) \subset E_+^* \cup (N^*\bar \Lambda) \cup \{\Phi_t(z,\xi) {~:~}z \in \bar \Lambda,~\xi \in \R \psi(z),~ t \geqslant 0\}.
\end{equation}
In particular, we have
$
 N^*\Lambda\cap \WF(Q(s) \iota_X [\bar \Lambda]) = \emptyset.
$
{
Indeed, since $r_+(z) > 0$, we have as before that $N^*\Lambda \cap E_+^* \subset \{0\}$; also $N^*\Lambda \cap N^*\bar \Lambda = \emptyset$ simply because $\Lambda \cap \bar \Lambda = \emptyset$; finally, the last term in the right-hand side of \eqref{eq:wfqss} can only intersect $N^*\Lambda$ in a trivial way by Remark \ref{rem:neq} and \eqref{eq:n*lambda}.} Therefore, the product $ [\Lambda] \wedge Q(s) \iota_X [\bar \Lambda]$ is well defined as a distribution {by \cite[Theorem 8.2.10]{hor1}}. As $s \mapsto Q(s)$ is meromorphic, so is the family $s \mapsto  [\Lambda] \wedge Q(s) \iota_X [\bar \Lambda]$, because the bound (\ref{eq:wf}) is satisfied locally uniformly in $s \in \C \setminus \mathrm{Res}(\Lie_X)$ (it follows from the proof of (\ref{eq:wf}) in \cite{dyatlov2016pollicott}).

In what follows, for any closed conical subset ${\Gamma} \subset T^*M_\delta$, we will denote 
$$
\mathcal{D}'^\bullet_{\Gamma}(M_\delta^\circ) = \{u \in \mathcal{D}'^\bullet(M_\delta^\circ) {~:~} \WF(u) \subset \Gamma\}
$$
endowed with its natural topology (see \cite[Definition 8.2.2]{hor1}).

\begin{prop}\label{prop:pairing}
If $\Re(s)$ is large enough, {the Poincar\'e series $\eta(s)$ converges, the pairing $\langle 1, [\Lambda] \wedge Q(s) \iota_X [\bar \Lambda] \rangle$ is well defined,} and it holds
$$
\eta(s) = -\langle 1, [\Lambda] \wedge Q(s) \iota_X [\bar \Lambda] \rangle.
$$
\end{prop}

{
\begin{rem}
As we mentioned above, we already know that the pairing $\langle 1, [\Lambda] \wedge Q(s) \iota_X [\bar \Lambda] \rangle$ makes sense by using the wavefront set properties of $Q(s)$ given in \cite{dyatlov2016pollicott}. However, we will prove below that this pairing is \textit{a priori} well defined provided that $\Re(s)$ is large enough (without using the results of \cite{dyatlov2016pollicott}) and we will see (using Lemma \ref{lem:distribution}) that this implies the convergence of the series $\eta(s)$.
\end{rem}
\begin{corr}
The function $s \mapsto \eta(s)$ extends meromorphically to the whole complex plane.
\end{corr}
\begin{proof}
We saw above that family $s \mapsto [\Lambda] \wedge Q(s) \iota_X [\bar \Lambda]$ extends meromorphically to the whole complex plane, and so does $s \mapsto \langle 1, [\Lambda] \wedge Q(s) \iota_X [\bar \Lambda] \rangle$. Thus Proposition \ref{prop:pairing} immediately implies the meromorphic continuation of $\eta$.
\end{proof}
}

\begin{proof}[Proof of Proposition \ref{prop:pairing}]
We fix $\varrho \in C^\infty(\R, [0,1])$ such that $\varrho(t) = 1$ for $t \geqslant \varepsilon$ and $\varrho(t) = 0$ for $t \leqslant \varepsilon/2$, where 
$
\varepsilon = \min(1/2,~  \inf_{\gamma \in \mathcal{G^\perp}}{\ell(\gamma)}).
$
For ${n} \in \N_{\geqslant 1}$ we take $\varrho_{{n}} \in C^\infty_c(\R_{>0}, [0,1])$ such that $\varrho_{{n}}(t) = 1$ for $t \leqslant n$ and $\varrho_{{n}}(t) = 0$ for $t \geqslant n +1$, and we set $\chi_n = \varrho_n\varrho $. Then we have
\begin{equation}\label{eq:pairinglimit}
\langle \mu, \chi_{{n}} {e}^{-s \cdot} \rangle \to \eta(s), \quad {n} \to +\infty,
\end{equation}
for $\Re(s) \gg 1$ by Lemma \ref{lem:distribution}, as $\varepsilon(\tau, z) = +1$ for any $z \in \Lambda$ such that $\varphi_{-\tau}(z) \in \bar \Lambda$ ({note that $\eta(s)$ could \textit{a priori} be infinite}).  Now, consider
$$
A_{{n}}(s) = \chi \int_{\R_+} \chi_{{n}}(t) {e}^{-ts} \iota_X \varphi_{-t}^* [\bar \Lambda] \dd t \in \mathcal{D}'^1(M_\delta^\circ),$$
{where $\chi \in C^\infty_c(M_\delta^\circ)$ is the cutoff function introduced in \S\ref{subsec:resolvent}}. Note that 
\begin{equation}\label{eq:convergencean(s)}
A_{{n}}(s) \to Q_\varrho(s) \iota_X [\bar \Lambda]
\end{equation}
in $\mathcal{D}'^1({M_\delta^\circ})$ when ${n} \to +\infty$ whenever $\Re(s)$ is large enough, where we set
$$
Q_\varrho(s) = \int_0^\infty\varrho(t){e}^{-ts} \varphi_{-t}^*\dd t.
$$
{Indeed, for any $\omega \in \Omega^\bullet(N)$ and $t \in \R$ we have $\|\varphi_{t}^*w\|_{\infty} \leqslant C \exp(C|t|) \|w\|_\infty$ (see for example \cite[Proposition A.4.1]{bonthonneau2015resonances}). 
In particular it holds $|\langle w, \iota_X \varphi_{-t}^*[\bar \Lambda]\rangle|\leqslant C \exp(C |t|) \|w\|_\infty$, and thus, if $\Re(s)$ is large enough, 
$$
A_n(s) \to \chi \int_{0}^\infty \varrho(t){e}^{-ts}\iota_X \varphi_{-t}^* [\bar \Lambda] \dd t
$$ 
as $n \to \infty$ by dominated convergence, and the integral defines a current of order $0$. Now the above integral coincides with $Q_\varrho(s)\iota_X [\bar \Lambda]$ as it follows by approximating $[\bar \Lambda]$ with smooth differential forms, and thus \eqref{eq:convergencean(s)} holds.} 

{Using \eqref{eq:pairinglimit} and Lemma \ref{lem:distribution}, we see that Proposition \ref{prop:pairing} will hold if we are able to show that the pairings $\langle [\Lambda], A_n(s)\rangle$ and $\langle 1, [\Lambda] \wedge Q(s) \iota_X [\bar \Lambda] \rangle$ are well defined, and that
\begin{equation}\label{eq:limitan}
\int_N [\Lambda] \wedge A_n(s) \to \langle 1, [\Lambda] \wedge Q(s) \iota_X [\bar \Lambda] \rangle
\end{equation}
as $n \to \infty$. To prove \eqref{eq:limitan} we will show that the convergence $A_{{n}}(s) \to Q_\varrho(s) \iota_X [\bar \Lambda]$ actually takes place in a finer topology than that of $\mathcal{D}'^1(M_\delta^\circ)$; this is the purpose of Lemma \ref{lem:seminorm} below. Then we will be able to conclude by noting that $\supp((Q(s) - Q_\varrho(s))\iota_X [\bar \Lambda])$ does not intersect $\supp([\Lambda])$.}

We will need the following result ({recall that $\psi$ is the connection form introduced in \S\ref{subsec:structuralforms}}).
\begin{lemm}\label{lem:sturm}
Let $\tau > 0$. Then there is $r > 0$ such that the following holds. For any $z \in M_\delta$ and $t \geqslant \tau$ such that $\varphi_t(z) \in M_\delta$, we have
\begin{equation}\label{eq:eststurm}
\bigl|\bigl\langle \varphi^*_t \psi (z), H(z) \bigr \rangle\bigr| \geqslant r { \bigl|\bigl\langle \varphi^*_t \beta(z), H(z) \bigr \rangle\bigr|.}
\end{equation}
Moreover we have $|\bigl\langle \varphi^*_t \beta(z), H(z) \bigr \rangle\bigr| \geqslant 1$ for any $t \geqslant 0$.
\end{lemm}

\begin{proof}
Let $z \in M_\delta$ and $\tau > 0$. Write $\varphi_t^*\beta(z) = a(t) \beta(z) + b(t) \psi(z)$ and $\varphi_t^*\psi(z) = c(t) \beta(z) + d(t) \psi(z)$ for $t \in \R$. We want to show that for $t \geqslant \tau$ one has {$|c(t)|\geqslant r |a(t)|$} for some $r > 0$. The structural equations (see \S\ref{subsec:structuralforms}) imply $\Lie_X \beta = \psi$ and $\Lie_X \psi = - \kappa \beta$. We thus obtain that $a$ and $b$ satisfy the following differential equation
\begin{equation}\label{eq:sturm}
y''(t) + \kappa(t) y(t) = 0
\end{equation}
where $\kappa(t) = \kappa(\varphi_t(z))$, with $a(0) = 1 = b'(0)$ and $a'(0) = 0 = b(0)$. Also $a'(t) = c(t)$ and $b'(t) = d(t)$. {It is easy to see that \eqref{eq:sturm} and the initial conditions imply $a'(t), a(t) > 0$ for $t > 0$. Thus we have $a'(t)a''(t) = - \kappa(t)a'(t)a(t) \geqslant k a'(t) a(t)$ where $k = \inf_{\Sigma_\delta} |\kappa|$. Integrating this, we get
$$
c(t)^2 = a'(t)^2 \geqslant k (a(t)^2 - 1).
$$
As $a'(t) > 0$ for $t > 0$ we have $a(t)^2 - 1 \geqslant a(\tau)^2 - 1$ for $t \geqslant \tau$, and thus it holds
$$
c(t)^2 \geqslant Ck a(t)^2, \quad t \geqslant \tau,
$$
where $C = 1 - 1 / a(\tau)^2 > 0$ (since $a(\tau) > a(0) = 1$). Setting $r = \sqrt{Ck}$ we obtain \eqref{eq:eststurm}. We conclude the proof of the lemma by noting that $a(t) \geqslant a(0) = 1$.}
\end{proof}

In what follows we set
$\displaystyle{
J_{{n}}(s) = \chi \int_{\R_+} (\chi_{{n}+1}(t) - \chi_{{n}}(t)) {e}^{-ts} \iota_X \varphi_{-t}^* [\bar \Lambda] \dd t \in \mathcal{D}'^1(M_\delta).
}
$
\begin{lemm}\label{lem:seminorm}
There exists a closed conical subset $\Gamma \subset T^*M_\delta$ not intersecting $N^*\Lambda$ such that for any {continuous} semi-norm $q$ on $\mathcal{D}'^1_\Gamma(M_\delta^\circ)$ (see \cite[Equation (8.2.2)]{hor1}), there is $C > 0$ such that
$$
q(J_{{n}}(s)) \leqslant C|s|^C{e}^{(C-\Re(s)){n}}, \quad n \geqslant 0.
$$
\end{lemm}

\begin{proof}
Let $w \in \Omega^2_c(M_\delta^\circ)$ supported in a small coordinate patch $U$ of some point $z_0 \in \Lambda$. Now by definition of $J_{{n}}$ it holds
$$
\langle w, J_{{n}}(s) \rangle = \int_0^\infty (\chi_{{n}+1}(t) - \chi_{{n}}(t)) {e}^{-ts} \left(\int_{\bar \Lambda} \iota_X \varphi_t^*w\right) \dd t.
$$
Let $\xi \in T^*_{z_0}M_\delta$. We identify $T^*U$ with $V \times \R^3$ for some neighborhood $V$ of $0 \in \R^3$. {Consider the Fourier transform $\langle w {e}^{i \langle \xi, \cdot \rangle}, J_n\rangle $; it holds}
\begin{equation}\label{eq:fourier}
\langle w {e}^{i \langle \xi, \cdot \rangle}, J_{{n}} \rangle = \int_{t=n}^{n+2}  \int_{u \in \bar \Lambda}  (\chi_{{n}+1}(t) - \chi_{{n}}(t)) {e}^{-ts} f(t,u) {e}^{i \langle \xi, \varphi_t(u)\rangle} \dd u \dd t,
\end{equation}
where $u$ is a choice of coordinate on $\bar \Lambda$ so that $\partial_u = H(u) \in T_u \bar \Lambda$, and $f$ is a smooth function satisfying for any ${k, \ell} \geqslant 0$
$$
|{\partial_t^k\partial_u^\ell} f(t, u)| \leqslant C_{{k,\ell}} {e}^{C_{{k,\ell}}t}, \quad t \geqslant 0, \quad u \in \bar \Lambda,
$$
for some $C_{{k,\ell}} > 0$\footnote{{The estimates on $f$ follow from the fact that $\|\iota_X \varphi_t^*w\|_{C^\ell} \leqslant C_\ell \exp(C_\ell |t|) \|w\|_{C^\ell}$ (see \cite[Proposition A.4.1]{bonthonneau2015resonances}). Thus 
$$\|\partial_t^k (\iota_X \varphi_t^* w)\|_{C^\ell} = \|\iota_X (\Lie_X)^k \varphi_t^* w\|_{C^\ell} \leqslant C_{k, \ell} \exp(C_{k, \ell} |t|) \|w\|_{C^{k + \ell}}.$$
Here we denoted by $\|\cdot\|_{C^\ell}$ the $C^\ell$ norm on $C^\infty(N, \wedge^\bullet T^*N$).}}; {note also that, as $w$ is supported in the coordinate patch $U$, we have that $f(t,u) = 0$ whenever $\varphi_t(u) \notin U$ and thus the expression ${e}^{i \langle \xi, \varphi_t(u)\rangle}$ is well defined on the support of $f$.} Now we have
\begin{equation}\label{eq:diffphase}
\partial_u \langle \xi, \varphi_t(u)\rangle = \langle \xi, \dd \varphi_t(u) H(u) \rangle, \quad \partial_t \langle \xi, \varphi_t(u)\rangle = \langle \xi, X(\varphi_t(u))\rangle.
\end{equation}
Let $\Gamma' = \{(z, \xi) \in T^*M_\delta {~:~}z \in \Lambda,~|\langle \xi, H({z}) \rangle|< \varepsilon |\xi|\}$ for $\varepsilon > 0$ small. Let $(z,\xi) \in \Gamma'$, $u \in \bar \Lambda$ and $t > 0$ such that $\varphi_t(u) = z \in U$. Then $t \geqslant \tau$ where $\tau > 0$ is a fixed  number which is smaller than the length of the shortest orthogeodesic. We decompose $\xi$ in the $(\alpha(z), \beta(z), \psi(z))$ basis as $\xi = \xi_\alpha \alpha(z) + \xi_\beta \beta(z) + \xi_\psi \psi(z)$. We have, since $\varphi_t^*\alpha(u) = \alpha(u)$,
$$
\begin{aligned}
\langle \xi,~\dd \varphi_t(u) H(u) \rangle &= \left\langle \varphi_t^*\bigl[\xi_\alpha \alpha(z) + \xi_\beta \beta(z) + \xi_\psi \psi(z)\bigr],~H(u)\right\rangle \\
&= \langle \xi_\beta \cdot \varphi_t^*\beta(u) + \xi_\psi \cdot \varphi_t^*\psi(u),~H(u) \rangle.
\end{aligned}
$$
Thus by Lemma \ref{lem:sturm} and the triangle inequality we have
$$
|\langle \xi,~\dd \varphi_t(u) H(u) \rangle| \geqslant r|\xi_\psi| - |\xi_\beta| ,
$$
for some $r > 0$ depending only on $\tau$ (indeed, the above inequality is obviously true even if $r|\xi_\psi| - |\xi_\beta| \leqslant 0$).
As $\xi \in \Gamma'$ we have 
\begin{equation}\label{eq:c(eps)}
|\xi_\beta|\leqslant \displaystyle{\frac{C\varepsilon}{1-\varepsilon}}(|\xi_\psi| + |\xi_\alpha|)
\end{equation}
for some $C > 0$. Therefore, we obtain
$$
|\partial_u \langle \xi, \varphi_t(u)\rangle| \geqslant (r - c(\varepsilon))|\xi_\psi| - c(\varepsilon)|\xi_\alpha|, \quad |\partial_t \langle \xi, \varphi_t(u)\rangle| = |\xi_\alpha|,
$$
where $c(\varepsilon) \to 0$ as $\varepsilon \to 0.$
Combining the estimate above with \eqref{eq:c(eps)} we obtain that there are $c,C > 0$ such that for any $\xi \in \Gamma'$ it holds
$$
C^{-1}\left\| \dd_{t,u}\left({e}^{i\langle \xi, \varphi_t(u)\rangle}\right)\right\| \geqslant (r-c(\varepsilon))|\xi_\psi| + (1 - c(\varepsilon)) |\xi_\alpha| \geqslant c(|\xi_\psi| + |\xi_\alpha|) \geqslant \frac{c}{2} |\xi|,
$$
provided that $\varepsilon$ is small enough. In particular we may apply the non-stationary phase method ({see for example \cite[Lemma 3.14]{zworski}}) to obtain that for any ${L} > 0$ we have $C_{{L}}$ such that 
$$
\left|\left\langle w{e}^{i\langle \xi, \cdot \rangle}, J_{{n}}(s) \right\rangle\right | \leqslant C_{{L}} |s|^L {e}^{(C_{{L}}-\Re(s)){n}} \langle \xi \rangle^{-{L}}, \quad \xi \in \Gamma',
$$
where $\langle \xi \rangle = \sqrt{1 + |\xi|^2}$ for some norm $|\cdot|$ on $T^*M_\delta$. By setting $\Gamma = \complement \Gamma'$, we obtain the sought result.
\end{proof}

This last result implies that for any {continuous} semi norm $q$ of $\mathcal{D}'^1_{\Gamma}(M_\delta^\circ),$ we have 
$$
q\left(A_{{n}}(s) - Q_\varrho(s)\iota_X[\bar \Lambda]\right) \to 0
$$
as ${n} \to +\infty$ if $\Re(s)$ is large enough (depending on $q$). For any finite set $\mathrm Q$ of continuous semi norms of $\mathcal{D}'_\Gamma(M_\delta^\circ)$ we define
$$
\mathcal{D}'^\bullet_{\Gamma, \mathrm Q}(M_\delta^\circ) = \{u \in \mathcal{D}'^\bullet(M_\delta^\circ)~:~q(u) < \infty,~q \in \mathrm Q\}.
$$
This set is endowed with the following topology: we say that $u_n \to u$ in $\mathcal{D}'^\bullet_{\Gamma, \mathrm Q}(M_\delta^\circ)$ if the convergence holds in $\mathcal{D}'^\bullet(M_\delta^\circ)$ and $\sup_n q(u_n) < \infty$ for any $q \in \mathrm Q$.
Then, since $[\Lambda]$ is compactly supported in $M_\delta^\circ$ and $\WF'([\Lambda]) \cap \Gamma = \emptyset$, we may reproduce the proof of the H\"ormander's theorem about product of distributions \cite[Theorem 8.2.10]{hor1} to obtain that there exists a finite set $\mathrm Q$ of semi norms of $\mathcal{D}'^\bullet_\Gamma(M_\delta^\circ)$ (which depends on $[\Lambda]$) such that the product $[\Lambda] \wedge u$ is well defined for any $u \in \mathcal{D}'^\bullet_{\Gamma, \mathrm Q}(M_\delta^\circ)$ and such that the map 
\begin{equation}\label{eq:continuousmapq}
\mathcal{D}'^\bullet_{\Gamma, \mathrm Q}(M_\delta^\circ) \to \mathcal{D}'^\bullet(M_\delta^\circ), \quad u \mapsto [\Lambda] \wedge u,
\end{equation}
is continuous. By Lemma \ref{lem:seminorm}, if $\Re(s)$ is large enough, the sequence $n \mapsto q(A_n(s))$ is bounded for any $q \in \mathrm Q$ and letting $n \to \infty$ yields $q(Q_\varrho(s)\iota_X [\bar \Lambda]) < \infty$ for every $q \in \mathrm Q$. Thus the products $[\Lambda] \wedge A_n(s)$ and $[\Lambda] \wedge Q_\varrho(s)\iota_X[\bar \Lambda]$ are well defined, and by continuity of the map \eqref{eq:continuousmapq}, we get
$$
[\Lambda] \wedge A_n(s) \to [\Lambda] \wedge Q_\varrho(s)\iota_X[\bar \Lambda], \quad n \to \infty,
$$
where the convergence holds in $\mathcal{D}'^\bullet(M_\delta^\circ).$ However, we have
$$
(Q(s)-Q_\varrho(s))\iota_X [\bar \Lambda] = \int_0^\infty (1 - \varrho(t)) {e}^{-ts}\varphi_{-t}^*\iota_X [\Lambda],
$$
and thus $\supp\bigl((Q(s)-Q_\varrho(s))\iota_X [\bar \Lambda]\bigr) \cap \supp([\Lambda]) = \emptyset$, since $\varphi_t(z) \notin \Lambda$ for $0 \leqslant t \leqslant \varepsilon$ and $z \in \bar \Lambda$. This yields 
$$
[\Lambda] \wedge Q_\varrho(s)\iota_X[\bar \Lambda] = [\Lambda] \wedge Q(s)\iota_X[\bar \Lambda],$$ and in particular, \eqref{eq:limitan} holds. This completes the proof of Proposition \ref{prop:pairing}.
\end{proof}

\section{Value of the Poincar\'e series at the origin}

In this section we show that $\eta(s)$ vanishes at $s = 0$.

\subsection{Behavior of $Q(s)$ at $s=0$}\label{subsec:behaviors=0}
{
By \cite[Theorem 2]{dyatlov2016pollicott}, we have the Laurent development
\begin{equation}\label{eq:laurentbad}
Q(s) = Y(s) + \sum_{j=1}^J  \frac{\chi (\Lie_X)^{j-1} \Pi \chi}{s^j}
\end{equation}
for some $J \geqslant 1$, where $s \mapsto Y(s)$ is holomorphic near $s=0$, and $\Pi : \Omega_c^\bullet(M_\delta^\circ) \to \mathcal{D}'^\bullet(M_\delta^\circ)$ is a finite rank projector satisfying 
\begin{equation}\label{eq:wfsupppi}
\supp(\Pi) \subset \Gamma_+ \times \Gamma_- \quad \text{ and } \quad \WF(\Pi) \subset E_{+}^* \times E_{-}^*.\end{equation}
Moreover, it holds $\mathrm{ran}(\Pi) = C^\bullet$ where
$$
C^{\bullet} = \left\{u \in \mathcal{D}'^\bullet_{E_+^*}(M_\delta^\circ)~:~\supp(u) \subset \Gamma_+,~(\Lie_X)^J u = 0\right\}.
$$
Elements of $C^\bullet$ are called generalized resonant states for $X$ (for the resonance $0$). A generalized resonant state $u$ is simply called a resonant state if $\Lie_X u = 0$. In what follows, we will set $\Omega^\bullet_0 = \Omega^\bullet_c(M_\delta^\circ) \cap \ker(\iota_X)$ and $C^\bullet_0 = C^\bullet \cap \ker(\iota_X)$. Since $\Lie_X \alpha = 0$ we have the decomposition
\begin{equation}\label{eq:decc}
C^\bullet = C^\bullet_0 \oplus \alpha \wedge C^{\bullet - 1}_0,
\end{equation}
and this decomposition is preserved by $\Pi.$ We now invoke a result of Hadfield (see \cite[Propositions 3, 4, 5]{hadfield2018zeta}) which implies that 
\begin{equation}\label{eq:hadfield}
C^0_0 = \{0\}, \quad C^2_0 = \{0\} \quad \text{and}  \quad \Lie_X (C^1_0)= \{0\}.
\end{equation}
In particular by \eqref{eq:decc} we have $\Lie_X \Pi = 0$ and thus \eqref{eq:laurentbad} yields
\begin{equation}\label{eq:laurent}
Q(s) = Y(s) + \frac{\chi \Pi \chi}{s}.
\end{equation}
Now let us decompose $\Pi$ as
$$
\Pi|_{\Omega_c^1(M_\delta^\circ)} = \sum_{j=1}^r u_j \otimes \beta_j,
$$
where $(u_j, \beta_j) \in \mathcal{D}'^{1}_{E_+^*}(M_\delta^\circ) \times \mathcal{D}'^{2}_{E_-^*}(M_\delta^\circ)$ satisfy $\supp u_j \subset \Gamma_+$ and $\supp \beta_j \subset \Gamma_-$ (such a decomposition necessarily exists by \eqref{eq:wfsupppi}). Then $\beta_j$ is a coresonant state for $X$, meaning that it is a resonant state for $-X$; applying \cite[Propositions 3, 4, 5]{hadfield2018zeta} for the vector field $-X$, we therefore obtain that $\beta_j = \alpha \wedge s_j$ for some coresonant state $s_j \in \mathcal{D}'^1_{E_-^*}(M_\delta^\circ)$ (indeed, we note that \eqref{eq:hadfield} gives $C^2 = \alpha \wedge C^1_0$, and we apply this to the vector field $-X$ instead of $X$}). Also, it follows from \cite[Lemma 6]{hadfield2018zeta} that the currents $u_j$ and $s_j$ are closed.

Summarizing the above results, we get
\begin{equation}\label{eq:pidecomp}
\Pi = \sum_{j=1}^r u_j \otimes \alpha \wedge s_j
\end{equation}
where $(u_j, s_j) \in \mathcal{D}'^1_{E_+^*}(M_\delta^\circ) \times \mathcal{D}'^1_{E_-^*}(M_\delta^\circ)$ satisfy
\begin{equation}\label{eq:propu}
\supp(u_j) \subset \Gamma_+, \quad \supp(s_j) \subset \Gamma_-, \quad \dd u_j = \dd s_j = 0, \quad \iota_X u_j = \iota_X s_j = 0.
\end{equation}
In particular, we have $\int_{M_{\delta}^\circ} \iota_X [\bar \Lambda] \wedge \alpha \wedge s_j = \int_{M_\delta^\circ} [\bar \Lambda] \wedge s_j$ and thus
$$
\langle [\Lambda], \Pi \iota_X [\bar \Lambda] \rangle ={\sum_{j=1}^{r} \left( \int_{M_\delta^\circ} [\Lambda] \wedge u_j \right) \left(\int_{M_\delta^\circ}  [ \bar \Lambda] \wedge s_j \right).
}
$$
Note that those products make sense since by Lemma \ref{lem:rpm} it holds
$$
E_+^* \cap N^*\Lambda \subset \{0\} \quad \text{and} \quad E_-^* \cap N^*\bar \Lambda \subset \{0\}.
$$
Let $\eta > 0$ and set $\Gamma_+^\eta = \{z \in M_\delta {~:~}\mathrm{dist}(z, \Gamma_+) < \eta\}$. By \cite[Lemma 6]{hadfield2018zeta}, we may find $f_j \in \mathcal{D}'(M_\delta^\circ)$ such that 
$$
\supp(f_j) \subset \Gamma_+^\eta, \quad \WF(f_j) \subset E_+^*, \quad \Lie_X f_j \in C^\infty_c(M_\delta^\circ),
$$
and such that $v_j = u_j-\dd f_j$ is smooth. Now since $[\Lambda]$ is compactly supported in $M_\delta^\circ$, we have
$
\displaystyle{\int_{M_\delta^\circ} [\Lambda] \wedge \dd f_j = 0}
$
(since $\dd [\Lambda] = 0$ as $\partial \Lambda = \emptyset$)
and thus
\begin{equation}\label{eq:uv}
\int_M [\Lambda] \wedge u_j = \int_M [\Lambda] \wedge v_j.
\end{equation}
Finally, take the coordinates $(\rho, {\tau}, \theta)$ given by Lemma {\ref{lem:coord}} near $\partial M = \{\rho = 0\}$ (here we assume for simplicity that $\partial M$ is connected but the exact same proof applies if it is not). We have $\Lambda = \{(0, {\tau}, +\pi / 2) {~:~} {\tau} \in \R/\ell_\star \Z\}$, and $\partial M_\delta = \{(\delta, {\tau}, \theta) {~:~}{\tau \in \R/\ell_\star \Z},~\theta \in \R/2\pi\Z\}$. Consider
$$
S_1 = \{(\rho, {\tau}, \pi / 2) {~:~}\rho \in [0, 2\delta / 3],~{\tau} \in \R/\ell_\star \Z\}
$$
and
$$
S_2 = \{(2 \delta / 3, \tau, \theta) {~:~} {\tau} \in \R/\ell_\star \Z,~ \theta \in [\pi / 2, 3\pi / 2]\}.
$$
Let $[S] = [S_1] + [S_2]$ and note that $\dd [S] = [\Lambda] - [W]$ by Stokes' theorem, where $$
W = \{(2 \delta / 3, {\tau}, 3\pi / 2) {~:~}\tau \in \R/\ell_\star \Z\}
$$
is the incoming normal set of ${\{\rho = 2\delta / 3\}}$, which is oriented with $\partial_\tau$. Now if $\eta$ is small enough
$$
\int [\Lambda] \wedge v_j = \int [W] \wedge v_j = 0
$$
since $W \cap \Gamma_+^\eta = \emptyset$ (indeed, $W$ is at positive distance of $\Gamma_+$) and $\supp v_j \subset \Gamma_+^\eta$. Thus we obtained that $s \mapsto \eta(s)$ has no pole at $s = 0$, and by Proposition \ref{prop:pairing} it holds
$$
\eta(0) = -\langle [\Lambda], Y(0) \iota_X [\bar \Lambda]\rangle.
$$

\subsection{Value at $s=0$}
In this subsection we prove Theorem \ref{thm:main}, that is, $\eta(0) = 0$.
Let us define 
$$
S_1 = \{\varphi_t(z) {~:~}0 \leqslant t \leqslant \delta / 2,~z \in \Lambda\}, \quad S_2 = \{ R_\theta(z) {~:~} 0 \leqslant \theta \leqslant \pi,~z \in \Lambda' \},
$$
where $R_\theta : M_\delta \to M_\delta$ is the rotation of angle $\theta$, and $\Lambda' = \varphi_{\delta / 2}(\Lambda)$. We orient $\Lambda'$ with the orientation of $\Lambda$. The manifold $S_1$ is oriented by declaring that $(\partial_t, \partial_\tau)$ is positively oriented (here $\partial_\tau$ is any positive basis of $T\Lambda$), and $S_2$ is oriented by declaring that the basis $(\partial_\theta, \partial_\tau')$, where $\partial_\tau'$ is any positive basis of $T\Lambda'$. Let $\Lambda'' = R_{\pi}(\Lambda')$. Note that (\ref{eq:wf}) implies, {by multiplication of wavefront sets (see \cite[Theorem 8.2.14]{hor1})},
\begin{equation}\label{eq:wfyo}
\WF(Y(0)\iota_X[\bar \Lambda])\subset  E_+^* \cup  N^*\bar \Lambda \cup \bigcup_{\substack{t \geqslant 0 \\ z \in \bar \Lambda}} \R \Phi_t(\psi(z)),
\end{equation}
where $Y(0)$ comes from (\ref{eq:laurent}). In what follows, we will set $Y = Y(0) \iota_X [\bar \Lambda]$ for simplicity. Since $E_+^* \cap N^*\Lambda \subset \{0\}$, and because $\delta$ is small, we have
\begin{equation}\label{eq:e+*n*}
E_+^* \cap N^*\Lambda' \subset \{0\}.
\end{equation}
Next, by analytic continuation and \eqref{eq:formularesolv}, it holds
$$
\supp(Y) \subset \overline{\bigcup_{t \geqslant 0} \varphi_t(\bar \Lambda)}.
$$
The right hand side of the above equation is disjoint from $\Lambda''$ by strict convexity of $M_\delta$. Thus 
\begin{equation}\label{eq:ylambda''}
\supp(Y) \cap \Lambda'' = \emptyset.
\end{equation}
Now for $z \in \Lambda$ and $t \in (0, \delta /2)$, we have
$
N^*_{\varphi_t(z)}(S_1 \setminus \partial S_1) \subset  \R \Phi_t(\psi(z))$ and  $N^*(S_2 \setminus \partial S_2) \subset \{\xi {~:~} \langle \xi, V \rangle = 0\}.$ 
In particular, by Lemmas \ref{lem:rpm} and \ref{lem:sturm}, we have
\begin{equation}\label{eq:wfysj}
\WF(Y) \cap N^*_{\varphi_t(z)}(S_j \setminus \partial S_j) = \emptyset, \quad j = 1,2.
\end{equation}
Finally, for $z \in \Lambda$, we have for $j = 1,2$, setting $z' = \varphi_{\delta/2}(z) \in \Lambda'$,
\begin{equation}\label{eq:wfsj}
\WF([S_j]) \cap T^*_{z'} M_\delta \subset \R\alpha(z') \oplus \R \Phi_{\delta/2}(\psi(z)).
\end{equation}
Combining \eqref{eq:wfyo}, \eqref{eq:e+*n*}, \eqref{eq:ylambda''}, \eqref{eq:wfysj} and \eqref{eq:wfsj}, we obtain that the intersection $\WF([S_j]) \cap \WF(Y)$ is contained in
$$
\left(\bigcup_{z \in \Lambda} \R \Phi_{\delta / 2}(\psi(z))\right) \cap \Biggl(~\bigcup_{\substack{t \geqslant 0 \\ \bar z \in \bar \Lambda}} \R \Phi_t(\psi(\bar z))\Biggr).
$$
However, by Lemma \ref{lem:sturm}, for any $\bar z \in \bar \Lambda$ and $t > 0$, we have $\Phi_{t}(\psi(\bar z)) \notin \R \psi(\varphi_t(\bar z))$. Therefore the above intersection is contained in the zero section and we get
\begin{equation}\label{eq:summarizing}
\WF([S_j]) \cap \WF(Y) = \emptyset, \quad j = 1,2,
\end{equation}
and in particular the product $[S_j] \wedge Y$ is well defined.
By Stokes' theorem, taking into account the orientations, we have
\begin{equation}\label{eq:dsj}
\dd[S_1] = [\Lambda'] - [\Lambda], \quad \dd [S_2] = [\Lambda''] - [\Lambda'].
\end{equation}
Then by (\ref{eq:inverse}) and the facts that $\dd [\bar \Lambda] = 0$, $[\dd, Y(0)] = 0$ (on $\{\chi = 1\}$) and $[\iota_X, Y(0)] = 0$ we have, by using \eqref{eq:laurentbad},
$$
\dd Y = \dd \iota_X Y(0) [\bar \Lambda]
= \Lie_X Y(0) [\bar \Lambda] 
= [\bar \Lambda] - \Pi([\bar \Lambda])
= [\bar \Lambda]
\quad \text{on} \quad \{\chi = 1\}
$$
as $\Pi( [\bar \Lambda]) = 0$ by \S\ref{subsec:behaviors=0} (we showed that $\int_{M_\delta^\circ} [\Lambda] \wedge u_j = 0$ for all $j$ but the same holds for $[\bar \Lambda]$ and $s_j$).
By Stokes' theorem, since $[S_j] \wedge Y$ is compactly supported in $M_\delta^\circ$ and $\dd Y = [\bar \Lambda]$ on $\{\chi = 1\} \supset \supp([S_j])$ ($j= 1, 2)$,
$$
\begin{aligned}
\int_{M_\delta} [\Lambda] \wedge Y &= -\int_{M_\delta} \dd [S_1] \wedge Y - \int_{M_\delta} \dd [S_2] \wedge Y + \int_{M_\delta} [\Lambda''] \wedge Y\\
&= \int_{M_\delta} [S_1] \wedge [\bar\Lambda] + \int_{M_\delta} [S_2] \wedge [\bar \Lambda] + \int_{M_\delta} [\Lambda''] \wedge Y.
\end{aligned}
$$
Finally, we have
$
\supp([S_j]) \cap \supp([\bar \Lambda]) = \emptyset
$
and by \eqref{eq:ylambda''}
we conclude that $$\eta(0) = \int_{M_\delta} [\Lambda] \wedge Y = 0.$$

\section{Poincar\'e series for geodesic arcs linking two points}
We fix $x \neq y \in \Sigma$. We consider
$$
\eta_{x,y}(s) = \sum_{\gamma : x \leadsto y} {e}^{-s \ell(\gamma)},
$$
where the sum runs over all the (oriented) geodesics joining $x$ to $y$. For $a \in \Sigma$ we will set $\Lambda_a = S_a\Sigma$. Note that 
$
T_z\Lambda_a = \R V(z)
$
for $z \in \Lambda_a$ (this follows from the definition of $V$ in \S\ref{subsec:structuralforms}), and we orient $\Lambda_a$ according to $V$.
 {In this context, we have the counterpart of Proposition \ref{prop:pairing}, as follows.
\begin{prop} For $\Re(s)$ large enough it holds
$$
\eta_{x,y}(s) = -\langle [\Lambda_x],~Q(s)\iota_X[\Lambda_y] \rangle.
$$
\end{prop}
Note that the above pairing makes sense, since we have the inclusion $\WF([\Lambda_x]) \subset N^*\Lambda_x$ which gives $\WF(Q(s)\iota_X[\Lambda_y]) \cap \WF([\Lambda_x]) = \emptyset$ (the emptiness of the last intersection can be seen by proceeding as in \S\ref{subsec:pairing}).
}
{
\begin{proof}[Sketch of the proof]
Using Remark \ref{rem:neq}, we see that for $t\geqslant 0$ and $z \in \Lambda_x$ such that $\varphi_{-t}(z) \in \Lambda_y$, one has the direct sum
$$
T_z M = T_z \Lambda_x \oplus \R X(z) \oplus \dd_{\varphi_{-t}(z)} \varphi_{t} (T_{\varphi_{-t}(z)} \Lambda_y).
$$
Moreover we check that the orientation of the right-hand side has the same orientation of $M$, again by Remark \ref{rem:neq}. Thus we have the counterpart of Lemma \ref{lem:distribution} in this context and for any $\chi \in C^\infty_c(\R_+)$ it holds
$$
\sum_{\gamma : x \rightsquigarrow y} \chi(\ell(\gamma)) = -\int_N [\Lambda_x] \wedge \int_{\R_+} \chi(t) \iota_X \varphi_{-t}^*[\Lambda_y] \dd t.
$$
Now we may proceed as in the proof of Proposition \ref{prop:pairing} to obtain the sought result, by approximating the function $t \mapsto \exp(-ts)$ with compactly supported functions of the form $t \mapsto \chi_n(t) \exp(-ts)$ and taking the limit as $n \to \infty$ (one should use appropriate versions of Lemmas \ref{lem:sturm} and \ref{lem:seminorm} to justify the convergence of the pairings).
\end{proof}
}
This result implies that $s \mapsto \eta_{x,y}(s)$ extends meromorphically to the whole complex plane, since $s \mapsto Q(s)$ does. To compute its value at zero, we will need the following
\begin{lemm}\label{lem:dang}
There exists $[S] \in \mathcal{D}'^{1}_c(M_\delta^\circ)$ with $\supp([S]) \subset M$,  $\WF([S]) \cap \WF([\Lambda_y]) = \emptyset$ and 
\begin{equation}\label{eq:lambdaxs}
[\Lambda_x] = {-}\frac{1}{\chi(\Sigma)} [\bar \Lambda] - \dd [S], \quad \int_{M_\delta} [S]\wedge [\Lambda_y] = \frac{1}{\chi(\Sigma)}.
\end{equation}
\end{lemm}
{
\begin{proof}
Here we adapt the arguments of \cite[\S6.3.2]{dang2020poincar}. Let $f_1 : \Sigma_\delta \to \R$ be a smooth function which coincides with $-\rho$ on $\{|\rho|\leqslant \delta\}$ (here $\rho$ is the coordinate given by Lemma \ref{lem:coord}) and such that $\dd f_1(y) \neq 0$ for any $y \in \partial \Sigma$. The set of Morse functions being open and dense \cite[Theorem 5.6]{laudenbach2011transversalite} in $C^\infty(\Sigma_\delta)$, we may find a Morse function $f_2 \in C^\infty(\Sigma_\delta)$ which is arbitrarily close to $f_1$ in the $C^1$ norm. Let $\chi_0 \in C^\infty(\Sigma_\delta, [0,1])$ such that $\chi_0 = 1$ near $\partial \Sigma$ and $\supp \chi_0 \subset \{|\rho|\leqslant \delta / 2\}$. Note that $\|\dd f_1\| = \|\dd \rho \|\geqslant C$ on $\{|\rho|\leqslant \delta\}$ for some $C > 0$, where $\|\cdot\|$ is any norm on $T^*\Sigma_\delta$. In particular, if $f_2$ is chosen close enough to $f_1$ in the $C^1$ topology, the function $f = \chi_0 f_1 + (1 - \chi_0) f_2$ is also a Morse function. Indeed, $f$ coincides with $f_2$ on $\Sigma_\delta \setminus \{|\rho|\leqslant \delta\}$; moreover $f - f_1 = (1 - \chi_0)(f_2 - f_1)$ so that $\| \dd f \| \geqslant C/2$ on $\{|\rho|\leqslant \delta\}$ whenever $f_2$ is close enough to $f_1$. Next we set
$$
S_f = \left\{\left(b, \frac{\nabla^g f(b)}{\|\nabla^g f(b)\|}\right)~:~b \in {\Sigma \setminus \mathrm{crit}(f)}\right\} \subset M_\delta,
$$
where $\nabla^gf \in C^\infty(\Sigma_\delta, T\Sigma_\delta)$ is the gradient of $f$ with respect to the metric $g$, and $\mathrm{crit}(f) = \{\dd f = 0\}$ is the set of critical points of $f$. We orient $S_f$ according to the orientation of $\Sigma$. Then by \cite[Lemma 6.7]{dang2020poincar} and Stokes' theorem, we obtain that the integration current $[S_f]$ extends to a current on $N$ and we have\footnote{Indeed, the boundary of $S_f$ (near $\partial M$) is $\bar \Lambda$. In the coordinates of Lemma \ref{lem:coord}, $\bar \Lambda = \{(0, \tau, -\pi / 2)\}$ is oriented by $\partial_\tau \equiv H$; as the outward normal pointing vector at $\partial \Sigma$ is $\partial_\rho$ and $(\partial_\rho, \partial_\tau)$ is negatively oriented, we obtain that the boundary term coming from Stokes' formula must be $-[\bar \Lambda]$.}
$$
\dd [S_f] = -[\bar \Lambda] - \sum_{a \in \mathrm{crit}(f)} (-1)^{\mathrm{ind}_f(a)} [\Lambda_a],
$$
where $\ind_f(a)$ is the index of $a$ as a critical point of $\nabla^g f$, that is, the number of negative eigenvalues of the linearization of $\nabla^g f$ at the point $a$. Note that (up to taking $f_2$ very close to $f_1$), we have $y \notin \mathrm{crit}(f).$ Thus for each $a\in \mathrm{crit}(f)$ we may find a path $\gamma_a : [0, 1] \to \Sigma$ joining $a$ to $x$, and avoiding $y$. Setting 
$$
\theta_a = \left\{(\gamma_a(t), v)~:~v \in S_{\gamma_a(t)}\Sigma,~t \in [0, 1]\right\}, \quad a \in \mathrm{crit}(f),
$$
we have $\dd [\theta_a] = [\Lambda_x] - [\Lambda_a].$ The Poincar\'e-Hopf formula (see \cite[p.35]{milnor1997topology}) yields\footnote{Note that $\nabla^g f$ is actually inward pointing, but this is irrelevant since $\dim \Sigma = 2$.}
$$
\sum_{a \in \mathrm{crit}(f)} (-1)^{\ind_f(a)}= \chi(\Sigma).
$$
In particular, by setting 
$$
\displaystyle{
[S] = \frac{1}{\chi(\Sigma)}\left([S_f] - \sum_a (-1)^{\ind_f(a)} \theta_a\right),
}
$$
we obtain the first part of \eqref{eq:lambdaxs}. For the second part, we first note that $\theta_a \cap \Lambda_y = \emptyset$. Moreover, $S_f$ intersects (transversally) $\Lambda_y$ only at the point $(y, \nabla^gf(y) / \|\nabla^g f(y)\|)$. Looking at the orientations we get that $\int_N [S_f]\wedge [\Lambda_y] = 1$ (this follows from \eqref{eq:volume} and the fact that $\Lambda_y$ is oriented according to $\psi$). Finally the wavefront set condition follows from the transversality of the intersection, and the lemma follows.
\end{proof}
}

Before proving Theorem \ref{thm:points}, we state a result about regularization of currents; this is a version of the de Rham regularization procedure (see \cite[\S15, Proposition 1]{chern2012differentiable}) which takes into account the wavefront sets.

{
\begin{lemm}\label{lem:reg}
There are operators
$$
R_\varepsilon : \mathcal{D}'^\bullet(N) \to \Omega^\bullet(N), \quad A_\varepsilon :  \mathcal{D}'^\bullet(N) \to \mathcal{D}'^\bullet(N), \quad \varepsilon \in [0, 1],
$$
such that for any $u \in \mathcal{D}'^\bullet(N)$, the following holds.
\begin{enumerate}[label = \normalfont(\roman*)]
\item\label{item:chain}
We have the identities
$
R_\varepsilon - \id = \dd A_\varepsilon + A_\varepsilon \dd
$
 and $[\dd, R_\varepsilon] = 0$; 
\item\label{item:support}
The supports of $R_\varepsilon u$ and $A_\varepsilon u$ are contained in the $C\varepsilon-$neighborhood of the support of $u$ for some $C > 0$ independent of $\varepsilon$;
\item\label{item:wf} For any closed conical neighborhood $\Gamma$ of $\WF(u)$ (i.e. $\WF(u) \subset \Gamma^\circ$), there is $\varepsilon_0 > 0$ such that 
$
\WF(A_\varepsilon u) \subset \Gamma
$
for each $\varepsilon \in [0, \varepsilon_0]$, and moreover the families $(A_\varepsilon u)_{\varepsilon \in [0, \varepsilon_0]}$ and $(R_\varepsilon u)_{\varepsilon \in [0, \varepsilon_0]}$ are bounded in $\mathcal{D}'^\bullet_{\Gamma}(N);$
\item\label{item:convergence} We have $R_\varepsilon u \to u$ in  $\mathcal{D}'^\bullet(N)$ as $\varepsilon \to 0.$
\end{enumerate}
\end{lemm}
}
\begin{proof}
{
Let $X_1, \dots, X_n$ be vector fields on $N$ generating $TN$ everywhere, and denote the associated flows by $\varphi_{1, t}, \dots, \varphi_{n, t}$ for $t \in \R$. Let $\varepsilon > 0$ and $\chi \in C^\infty_c(\R^n, [0,1])$ such that $\chi = 1$ near $0$ and $\int_{\R^n} \chi(\mathbf{t}) \dd \mathbf{t} = 1$. For $u \in \mathcal{D}'^\bullet(N)$ we define
$$
R_\varepsilon u = \int_{\R^n} \chi(\mathbf t) \varphi_{1, \varepsilon t_1}^*\cdots \varphi_{n, \varepsilon t_n}^* u~ \dd \mathbf t.
$$
For $z \in N$ and $\mathbf{t} = (t_1, \dots, t_n) \in \R^n$ we will set $\Psi_{\varepsilon, z}(\mathbf{t}) = \Phi_{\varepsilon \mathbf{t}}(z)$ where
$$
\Phi_{\mathbf{t}} = \varphi_{n, t_n} \circ \cdots \circ \varphi_{1,  t_1}.
$$
We claim that $R_\varepsilon u$ is smooth. Indeed, if $u$ is a $0$-form, then we have
$$
R_\varepsilon u(z) = \langle \chi, \Psi_{\varepsilon, z}^*u\rangle, \quad z \in N,
$$
where the pairing is taken in $\R^n$. Indeed, this formula is true for $u$ smooth and thus it remains true for any distribution $u$ by continuity of the pullback $\Psi_{\varepsilon, z}^* : \mathcal{D}'^\bullet(N) \to \mathcal{D}'^\bullet(\R^n)$ (this follows from \cite[Theorem 6.1.2]{hor1} as $\Psi_{\varepsilon, z}$ is a submersion $\R^n \to N$ whenever $\varepsilon > 0$, since the vector fields $X_j$ generate $TN$). In particular $R_\varepsilon u$ is smooth, because $\Psi_{\varepsilon, z}$ depends smoothly on the variable $z$. If $u \in \Omega^k(N)$, we write locally $u = \sum_{\ell} u_\ell e_\ell$ for some basis $(e_\ell)$ of $\wedge^k T^*N$; writing $\Phi_{\varepsilon \mathbf{t}}^*e_\ell = \sum_j \alpha_{\ell, j}(\mathbf t) e_j$ we get by what precedes
$$
R_\varepsilon u(z) = \sum_{\ell, j} \langle \alpha_{\ell, j} \chi, \Psi_{\varepsilon,z}^*u_\ell \rangle e_j(z),
$$
and thus $R_\varepsilon u$ is smooth.
It is immediate to see that $R_\varepsilon u \to u$ in the distributional sense as $\varepsilon \to 0$, which is point \ref{item:convergence}. Next, note that
$$
(R_\varepsilon - \id)u = \int_{0}^\varepsilon \partial_r \left(\int_{\R^n} \chi(\mathbf t) \Phi_{r\mathbf{t}}^*u~ \dd \mathbf t\right)\dd r.
$$
By Cartan's formula one has
$
\partial_r \Phi_{r \mathbf{t}}^* = \dd B_{r \mathbf t} + B_{r \mathbf{t}} \dd
$
where $B_\mathbf{t} : \mathcal{D}'^\bullet(N) \to \mathcal{D}'^{\bullet - 1}(N)$ is defined by
$$
B_{\mathbf t} = \sum_{j=1}^n \varphi_{1, t_1}^* \cdots \varphi_{j, t_j}^* \iota_{X_j} \varphi_{j+1, t_{j+1}}^* \cdots \varphi_{n ,t_n}^*, \quad \mathbf t = (t_1, \dots, t_n) \in \R^n.
$$
Thus by setting 
$
\displaystyle{
A_\varepsilon = \int_{0}^\varepsilon \int_{\R^n} \chi(\mathbf t)B_{r \mathbf t} ~\dd \mathbf t \dd r
}
$
we obtain \ref{item:chain}. Property \ref{item:support} is clear and thus it remains to show that \ref{item:wf} holds. Let $u \in \mathcal{D}'^\bullet(N)$ and let $\Gamma$ be a conical neighborhood of $\WF(u)$. Take $(z_0, \xi_0) \in \complement \Gamma$, and a conical neighborhood $\Gamma_{0}$ of $\xi_0$ such that $\overline \Gamma_0 \cap \Gamma = \emptyset.$ Let $\varepsilon_0 > 0$ small enough so that $\Phi_{\varepsilon \mathbf t}^* (\Gamma_0) \cap \Gamma = \emptyset$ for any $\varepsilon \in [0, \varepsilon_0]$ and $\mathbf t \in \supp \chi$. Let $\omega \in \Omega^\bullet(N)$ be supported in a coordinate chart near $z_0$; we have for $\xi \in \Gamma_0$
$$
\begin{aligned}
\int_N \omega {e}^{i\langle \xi, \cdot \rangle} \wedge A_\varepsilon u 
&= \int_0^\varepsilon \int_{\R^n} \chi(\mathbf t) \left(\int_N  \omega {e}^{i \langle \xi, \cdot \rangle} \wedge B_{r \mathbf t}u\right) \dd \mathbf t \dd r.
\end{aligned}
$$
Thanks to the expression of $B_{r \mathbf t}$ one can see that the right hand side can be written as
\begin{equation}\label{eq:righthand}
\int_0^\varepsilon \int_{\R^n} \chi(\mathbf t)\left( \int_N f(\mathbf t,r) {e}^{i\langle \Xi(\mathbf t,r), \cdot\rangle} \wedge u\right) \dd \mathbf t \dd r
\end{equation}
where $f(\mathbf t,r)$ is a smooth function depending smoothly on $(\mathbf t,r)$ and $\Xi : \R^n \times [0, \varepsilon] \to T^*N$ is a smooth function satisfying $\Xi(\mathbf t, r) \in \Gamma_0$ on $\supp \chi$ and $|\Xi(\mathbf t, r)|\geqslant C |\xi|.$ Since $\Gamma$ does not intersect $\overline \Gamma_0$, the integral on $N$ in \eqref{eq:righthand} decays rapidly (i.e. faster than $\langle \xi \rangle^{-k}$ for any $k \geqslant 0$) as $\xi \to \infty$ and $\xi \in \Gamma_0$, with speed decay which is locally uniform with respect to $(\mathbf{t}, r) \in \supp\chi \times [0, \varepsilon]$. The result follows.
}
\end{proof}

\begin{proof}[Proof of Theorem \ref{thm:points}]
By \eqref{eq:uv} and {Lemma \ref{lem:dang}}, we have
$$
\begin{aligned}
\int_{M_\delta}[\Lambda_x] \wedge u_j &= \int_{M_\delta} [\Lambda_x] \wedge v_j \\
&= {-}\frac{1}{\chi(\Sigma)} \int_{M_\delta} [\bar \Lambda] \wedge v_j - \int_{M_\delta} \dd [S] \wedge v_j \\
&= 0
\end{aligned}
$$
{since $\supp(v_j) \subset \Gamma_+^\eta$ with $\Gamma_+^\eta \cap \bar \Lambda = \emptyset$ and $\dd v_j = \dd(u_j - \dd f_j) = 0$ by \eqref{eq:propu}.}
This shows that $\eta_{x,y}(s)$ has no pole at $s=0$ and that $\eta_{x,y}(0) = {-}\langle [\Lambda_x], Y(0) \iota_X [\Lambda_y] \rangle.$
Now since $Y(0)\iota_X [\Lambda_y]$ is compactly supported in $M_\delta^\circ$ we may view this pairing as a pairing on $N$, so that
$$
\eta_{x,y}(0) = -\int_{N} [\Lambda_x] \wedge Y(0) \iota_X [\Lambda_y].
$$
{
From \eqref{eq:inverse} we deduce that $\dd Y(0) \iota_X [\Lambda_y] = [\Lambda_y] + u$ for some current $u$ supported far from $M_{\delta /2}$. Let $\varepsilon > 0$ small. As $\dd [ \Lambda_x] = 0$, we have by Lemma \ref{lem:reg} that
$
[\Lambda_x] = R_\varepsilon [\Lambda_x] - \dd A_\varepsilon [\Lambda_x],
$
with $\WF(A_\varepsilon [\Lambda_x])$ close to $\WF([\Lambda_x])$; thus we may compute
$$
\begin{aligned}
\int_{N} [\Lambda_x] \wedge Y(0) \iota_X [\Lambda_y] 
&= \int_N R_\varepsilon [\Lambda_x] \wedge Y(0) \iota_X [\Lambda_y] - \int_N \dd A_\varepsilon [\Lambda_x] \wedge ([\Lambda_y] + u) \\
& = -\int_N \dd R_\varepsilon [S] \wedge Y(0) \iota_X [\Lambda_y] - \frac{1}{\chi(\Sigma)} \int_N R_\varepsilon[\bar \Lambda] \wedge  Y(0) \iota_X [\Lambda_y] \\
& \hspace{2cm}- \int_N \dd A_\varepsilon [\Lambda_x] \wedge ([\Lambda_y] + u),
\end{aligned}
$$
where we used Lemma \ref{lem:dang} in the last equality.
By point \ref{item:support} of Lemma \ref{lem:reg}, the second integral vanishes for small $\varepsilon$ since $\bar \Lambda \cap \supp(Y(0) \iota_X [\Lambda_y]) = \emptyset$; the third one also vanishes to zero as $\supp([\Lambda_x]) \cap \supp([\Lambda_y]) = \emptyset.$ Finally the first one writes
$$
 \int_N R_\varepsilon [S] \wedge ([\Lambda_y] + u),
$$
and thus it converges to $1/ \chi(\Sigma)$ as $\varepsilon \to 0$ thanks to the second equation of \eqref{eq:lambdaxs} and points \ref{item:support}, \ref{item:wf} and \ref{item:convergence} of Lemma \ref{lem:reg} (since $\supp([S]) \cap \supp(u) = \emptyset$). This concludes the proof of Theorem \ref{thm:points}.}
\end{proof}

\bibliographystyle{alpha}
\bibliography{bib}

\end{document}